\documentclass[11pt]{amsart}
\usepackage{amssymb}
\usepackage[margin=1in]{geometry}
\usepackage[colorlinks]{hyperref}
\newtheorem{theorem}{Theorem}[section]
\newtheorem{lem}[theorem]{Lemma}

\theoremstyle{definition}

\newtheorem{example}[theorem]{Example}

\theoremstyle{remark}
\newtheorem{remark}[theorem]{Remark}

\numberwithin{equation}{section}
\begin{document}

\newcommand{\spacing}[1]{\renewcommand{\baselinestretch}{#1}\large\normalsize}
\spacing{1.14}

\title{Left invariant lifted $(\alpha,\beta)$-metrics of Douglas type on tangent Lie groups}

\author{Masumeh Nejadahmad}
\address{Masumeh Nejadahmad, Department of Mathematics, Isfahan University of Technology.\\ E-Mail: \tt{masumeh.nejadahmad@math.iut.ac.ir}}

\author{Hamid Reza Salimi Moghaddam}
\address{Hamid Reza Salimi Moghaddam, Department of Mathematics, Faculty of  Sciences, University of Isfahan, Isfahan, 81746-73441-Iran.\\ E-Mail: \tt{hr.salimi@sci.ui.ac.ir and salimi.moghaddam@gmail.com}}


\date{\today}
\begin{abstract}
In this paper we study lifted left invariant $(\alpha,\beta)$-metrics of Douglas type on tangent Lie groups. Let $G$ be a Lie group equipped with a left invariant $(\alpha,\beta)$-metric of Douglas type $F$, induced by a left invariant Riemannian metric $g$. Using vertical and complete lifts, we construct the vertical and complete lifted $(\alpha,\beta)$-metrics $F^v$ and $F^c$ on the tangent Lie group $TG$ and give necessary and sufficient conditions for them to be of Douglas type. Then, the flag curvature of these metrics are studied. Finally, as some special cases, the flag curvatures of $F^v$ and $F^c$ in the cases of Randers metrics of Douglas type, and Kropina and Matsumoto metrics of Berwald type are given.
\end{abstract}

\maketitle

\section{\textbf{Introduction}}
Tangent bundles of differentiable manifolds have great importance in many fields of mathematics and physics. The study of Riemannian geometry of tangent bundles goes back to the fundamental paper \cite{Sasaki} written by Sasaki published in 1958. He used vertical and horizontal lifts to show that any Riemannian manifold $(M,g)$ induces a Riemannian metric on $TM$. Yano and Kobayashi replaced the horizontal lift with complete lift and studied many geometric properties of such lifted metrics (see \cite{Yano-Kobayashi 1}, \cite{Yano-Kobayashi 2} and \cite{Yano-Kobayashi 3}). Asgari and the second author studied the Riemannian geometry of lifted invariant Riemannian metrics induced on $TG$ by using vertical and complete lifts (\cite{Asgari-Salimi 1} and \cite{Asgari-Salimi 2}).\\
Using the lifted invariant Riemannian metrics together with vertical and complete lifts, they constructed two types of left invariant Randers metrics on the tangent bundle of Lie groups and studied their flag curvature in the case of Berwald metric (see \cite{Asgari-Salimi 3}). In this work, using the same way we build left invariant $(\alpha,\beta)$-metrics on tangent Lie groups. We give a necessary and sufficient condition for lifted $(\alpha,\beta)$-metrics to be of Douglas type and compute their flag curvatures.\\
Now we give some preliminaries about vertical and complete lifts and also Finsler geometry. \\

Let $M$ be an $m$-dimensional smooth manifold. Suppose that $TM$ is its tangent bundle. Any vector field $X$ on $M$ defines two types of (local) one-parameter groups of diffeomorphisms on $TM$ as follows:
 \begin{eqnarray}
\begin{array}{lr}
\phi_{t}(y) := (T_{x}\varphi_{t})(y),\ \ \ \ \forall x \in M, \forall y \in T_{x}M, &\nonumber\\
\psi_{t}(y) := y + t X(x), &
\end{array}
\end{eqnarray}
where $\varphi_{t}$ is the flow generated by the vector field $X$ on $M$. The maps $\phi_{t}$ and $\psi_{t}$ are the infinitesimal generators of one-parameter groups of diffeomorphisms. The vector fields corresponding to these one-parameter groups are called the complete lift (denoted by $X^{c}$) and vertical lift (denoted by $X^{v}$) of $X$, respectively. \\
For a local coordinates system $(U,(x^1\cdots,x^n))$ of $M$, we denote the induced local coordinates system on $TM$ by $(\pi^{-1}(U), (x^1\cdots,x^n;y^1\cdots,y^n))$, where $\pi : TM \rightarrow M$ is the projection map. Assume that $X$ is a vector field on $M$ with local representation $X\mid_{U} =\Sigma _{i=1}^{n} \xi^{i}\frac{\partial}{\partial x^{i}} $. Then its vertical and complete lifts in terms of local coordinates system $(x^{i} , y^{i})$ are as follows:
\begin{eqnarray}
\begin{array}{lr}
(X\mid _{U})^{v} = \Sigma _{i=1}^{n} \xi ^{i} \frac{\partial}{\partial y^{i}}, &\nonumber\\
(X\mid _{U})^{c} = \Sigma _{i=1}^{n} \xi ^{i} \frac{\partial}{\partial x^{i}} + \Sigma _{i,j=1}^{n} \frac{\partial \xi^{i}}{\partial x^{j}} y^{j} \frac{\partial}{\partial y^{i}}.
\end{array}
\end{eqnarray}
The Lie brackets of vertical and complete lifts of two vector fields $X$ and $Y$ satisfy the following equations (for more details see \cite{Yano-Ishihara, Hindeleh}),
\begin{equation}\label{Lie brackets of compelet and vertical}
    [ X^{v},Y^{v}] =0,\ \ \ \ \ [X^{c},Y^{c}]= [X,Y]^{c},\ \ \ \ \ [X^{v},Y^{c}] = [X,Y]^{v}.
\end{equation}
Assume that $G$ is a real $n$-dimensional connected Lie group with multiplication map $ \mu : G \times G \rightarrow G, (x,y) \mapsto xy $ and identity element $e$. Let $ l_{y} : G \rightarrow G,~x\mapsto yx $ and $ r_{y} : G\rightarrow G,~x \mapsto xy $ be left and right translations, respectively. Then we can see for all $ v \in T_{g}G$ and $w \in T_{h}G $ the tangent map
\begin{eqnarray}
\begin{array}{lr}
T\mu : T(G \times G) \cong TG \times TG \rightarrow TG  &\nonumber\\
(v,w) \mapsto T\mu (v,w) = T_{h}l_{g}w + T_{g}r_{h}v,~~ &
\end{array}
\end{eqnarray}
defines a Lie group structure on $TG$ with identity element $ 0_{e} \in T_{e}G $ and the inversion map $ T\iota $, where $ \iota $ is the inversion map of G (see \cite{Hilgert-Neeb}).\\

In \cite{Hindeleh}, it is shown that if $X$ is a left invariant vector field on $G$ then $ X^{c} $ and $ X^{v} $ are left invariant vector fields on $TG$. Therefore, for any left invariant Riemannian
metric $g$ on $G$, we can define a left invariant Riemannian metric $ \tilde{g} $ on TG as follows:
\begin{equation}\label{lifted Riemannian metric}
    \tilde{g} (X^{c},Y^{c}) = g(X,Y),~~~\tilde{g}(X^{v},Y^{v}) = g(X,Y),~~~\tilde{g}(X^{c},Y^{v}) = 0,
\end{equation}
where $X$ and $Y$ are arbitrary vector fields on $G$. In this work, we study the curvature of  left invariant $(\alpha , \beta)$-metrics of Douglas type on $TG$, where $\alpha$ is induced by a lifted left invariant Riemannian metric $\tilde{g}$. \\
A special type of Finsler metrics which belongs to the family of $(\alpha , \beta)$-metrics is Randers metric.
G. Randers introduced this family of Finsler metrics in his paper \cite{Randers} on general relativity, in 1941. These metrics have been used in more physical problems. For example, in four-dimensional case,  they are used in computation of the Lagrangian function of a test electric charge in the electromagnetic and gravitational fields (see \cite{Asanov, Ingarden, Landau-Lifshitz}).\\
A generalization of Randers metrics are $(\alpha , \beta)$-metrics which introduced by M. Matsumoto, in \cite{Matsumoto1}. These metrics are important and interesting types of Finsler metrics. \\
Let $(M, g)$ be a Riemannian manifold and $\beta$ be a $1$-form on $M$. Assume that $\alpha(x,y)=\sqrt{g_{ij} y^{i}y^{j}}$ and $\phi : (-b_{0} , b_{0}) \rightarrow (\mathbb{R})^{+} $ is a smooth map. It is shown that $ F=\alpha \phi (\frac{\beta}{\alpha})$  is a Finsler metric on $M$, which is called an $(\alpha , \beta)$-metric, if and only if $ \| \beta \|_{\alpha} < b_{0} $ and $ \phi = \phi(s) $ satisfies the following conditions (see \cite{Chern-Shen}):
\begin{equation}\label{inequality}
    \phi(s) - s\phi'(s) +(b^{2}-s^{2}) \phi''(s)>0,~~\mid s\mid \leq b < b_{0}.
\end{equation}
As some special cases, if $ \phi(s)=1+s$, $ \phi(s)=\frac{1}{s}$ or $ \phi(s)= \frac{1}{1-s} $, then we obtain three famous classes of Finsler metrics, which are called Randers metric $ \alpha + \beta $, Kropina metric $\frac{\alpha^{2}}{\beta}$ and Matsumoto metric $\frac{\alpha^{2}}{\alpha - \beta}$, respectively \cite{Chern-Shen}. \\
It is easy to see that for an arbitrary $1$-form $\beta$ on a Riemannian manifold $(M,g)$, there exists a unique vector field $X$ on $M$ such that for all $x \in M$ and $y\in T_{x}M$ we have
\begin{equation}
    g(y,X(x)) = \beta(x,y).
\end{equation}
This notation is very useful for construction left invariant $ (\alpha , \beta) $-metrics on Lie groups. If $(G,g)$ is a left invariant Riemannian metric and $X$ is a left invariant vector field on $G$ such that $ \parallel X \parallel_{\alpha}< b_{0} $, then the $(\alpha , \beta)$-metric which is defined as above is left invariant (see \cite{Deng-book} and \cite{Deng-Hosseini-Liu-Salimi}).\\
In this article we will study the flag curvature of some special Finsler metrics. This quantity is an important concept in Finsler geometry which is defined by
\begin{equation}\label{flag curvature formula}
    K(P,y) = \frac{g_{y}(R(u,y)y,u)}{g_{y}(y,y) g_{y}(u,u) - g_{y}^{2}(u,y)},
\end{equation}
where $P =\textsf{span}\{u,y\} $, $ g_{y}(u,v)=\frac{1}{2}\frac{\partial^{2}}{\partial s \partial t}F^{2}(y+su+tv)\mid _{s=t=0} $ is the fundamental tensor, $ R(u,y)y = \nabla_{u}\nabla_{y}y - \nabla_{y}\nabla_{u}y - \nabla_{[u,y]}y  $ is the curvature tensor and $ \nabla $ is the Chern connection of F (see \cite{Bao-Chern-Shen},\cite{Chern-Shen})\\
Suppose that $F$ is a Finsler metric on a smooth $n$-dimensional manifold $M$. In a standard local coordinates system of $TM$, the spray coefficients of $F$ are defined by
\begin{equation}\label{Spray coefficients}
    G^{i}(x,y):=\frac{1}{4}g^{il}([F^{2}]_{x^{m}y^{l}}y^{m}-[F^{2}]_{x^{l}}),~~i=1,\cdots, n, x\in M, y \in T_{x}M.
\end{equation}
The Finsler metric $F$ is called a Douglas metric if the spray coefficients $G^{i}$ satisfy the following relation,
\begin{equation}\label{Douglas condition}
    G^{i} = \frac{1}{2}\Gamma^{i}_{jk}(x)y^{j}y^{k}+P(x,y) y^{i},
\end{equation}
where $P(x,y)$ is a local positively homogeneous function of degree one on $TM$ and $F$ is called of Berwald type if $ P(x,y)=0 $ (see \cite{Bacso-Matsumoto} and \cite{Chern-Shen}). For an $ (\alpha , \beta) $-metric F it is well known that it is of Berwald type if and only if the 1-form $ \beta $ is parallel with respect to the Levi-Civita connection of $ \alpha $ \cite{Bao-Chern-Shen}. \\
In the last decade, many geometric properties of Lie groups equipped with left invariant Finsler metrics, or homogeneous spaces together with invariant Finsler metrics have been studied (for example see \cite{Deng-Hou, Deng-Hu} and \cite{Salimi1, Salimi2, Salimi3}). In this work we use the following theorem which is proved by Liu and Deng in \cite{Liu-Deng}.
\begin{theorem}\label{Liu-Deng theorem}
Assume that $F=\alpha \phi(\frac{\beta}{\alpha})$ is a homogeneous $(\alpha , \beta)$-metric on $G/H$. Then $F$ is a Douglas metric if and only if either $F$ is a Berwald metric or $F$ is a Douglas metric of Randers type.
\end{theorem}
\section{\textbf{Lifting of $ (\alpha , \beta) $-metrics on Tangent Bundles}}
Let $G$ be a Lie group equipped with a left invariant Riemannian metric $g$. In \cite{Asgari-Salimi 1}, Asgari and the second author proved that for any $ X , Y \in \mathfrak{g}= Lie(G) $ the Levi-Civita connection of the lifted left invariant metric $\tilde{g}$ on $TG$ can be computed by the following equations:

\begin{equation}\label{Levi-Civita connection}
\left\{ \begin{array}{l}
 \tilde{\nabla} _{X^{c}}Y^{c}= (\nabla_{X}Y)^{c} \\
 \tilde{\nabla} _{X^{v}}Y^{v} = (\nabla_{X}Y - \frac{1}{2}[ X,Y ] )^{c} \\
 \tilde{\nabla} _{X^{c}}Y^{v} = (\nabla_{X}Y + \frac{1}{2} ad^{\ast}_{Y}X)^{v} \\
 \tilde{\nabla} _{X^{c}}Y^{v} = (\nabla_{X}Y + \frac{1}{2} ad^{\ast}_{Y}X)^{v}\end{array} \right.
\end{equation}
Suppose that $F$ is a left invariant $(\alpha,\beta)$-metric on $G$ defined by
\begin{equation}\label{me}
 F = \sqrt{g(y,y)} \phi(\frac{g(X(x),y)}{\sqrt{g(y,y)}})
\end{equation}
where $g$ and $X$ are a left invariant Riemannian metric and a left invariant vector field on $G$, respectively. Then by using vertical and complete lifts, we can define two types of left invariant Finsler metrics on $TG$ as follows:
\begin{equation}\label{complete lifted Finsler}
    F^{c}((x,y),\tilde{z})= \sqrt{\tilde{g}(\tilde{z},\tilde{z})} \phi(\frac{\tilde{g}(X^{c}(x,y),\tilde{z})}{\sqrt{\tilde{g}(\tilde{z},\tilde{z})}}),
\end{equation}

\begin{equation}\label{vertical lifted Finsler}
    F^{v}((x,y),\tilde{z})= \sqrt{\tilde{g}(\tilde{z},\tilde{z})} \phi(\frac{\tilde{g}(X^{v}(x,y),\tilde{z})}{\sqrt{\tilde{g}(\tilde{z},\tilde{z})}}),
\end{equation}
where $ x \in G $, $ y \in T_{x}G $ and $ \tilde{z} \in T_{(x,y)}TG$.\\
Since $ \parallel X^{c}\parallel_{\tilde{g}} =\parallel X^{v}\parallel_{\tilde{g}} = \parallel X\parallel_{g} <b_{0}$, $ F^{c} $ and $ F^{v} $ are left invariant $ (\alpha , \beta) $-metrics on $TG$.
In this section we suppose that $F$, $ F^{c} $ and $ F^{v} $ are defined as above.\\
We give a necessary and sufficient condition for $ F^{c}$ and $ F^{v}$ to be of  Douglas type.
\begin{lem}\label{fd}
Let $F$ be an arbitrary left invariant $ (\alpha , \beta)$-metric defined by
\ref{me}, where $g$ is a left invariant Riemannian metric and $X$ is a left invariant vector field on a Lie group $G$. $F$ is of Douglas type if and only if $F^{c}$ is of Douglas type.
\end{lem}
\begin{proof}
Suppose that $F$ is a Douglas metric. \ref{Liu-Deng theorem} shows that $F$ is a Berwald metric or a Douglas metric of Randers type. If $F$ is of Berwald type then proposition 5.5 of \cite{An-Deng} shows that for all $ Y, Z\in \mathfrak{g}$ we have
\begin{equation}
     g ([ Z ,Y] , X) = g (ad^{\ast}_{Y}X , Z)=0.
\end{equation}
So $ ad^{\ast}_{Y}X=0 $. Now the formula \ref{Levi-Civita connection} proves $ \tilde{\nabla} _{Y^{c}}X^{c} = \tilde{\nabla} _{Y^{v}}X^{c} =0$ which means that $ F^{c}$ is of Berwald type. If $F$ is a Douglas metric of Randers type, then by theorem 3.2 of \cite{An-Deng}, for all $Y , Z \in \mathfrak{g}$, we have $ g ([ Z ,Y] , X) =0$. On other hand we have the following relations.
\begin{equation}
    \tilde{g}([Z^{c},Y^{c}],X^{c})= g ([ Z ,Y] , X),~~\tilde{g}([Z^{v},Y^{c}],X^{c})=0,~~\tilde{g}([Z^{v},Y^{v}],X^{c})=0.
\end{equation}
Thus the same theorem says $ F^{c} $ is a Douglas metric of Randers type.\\
Conversely, let $F^{c}$ be of Douglas type. If $F^{c}$ is a Berwald metric then for any $ Y \in \mathfrak{g}$ we have
$ \tilde{\nabla} _{Y^{c}}X^{c} = \tilde{\nabla} _{Y^{v}}X^{c} =0$. So for any $ Y \in \mathfrak{g}$ we have $ \nabla _{Y}X =0 $, which means that $F$ is of Berwald type. If $ F^{c} $ is a Randers metric of Douglas type then $g ([ Z ,Y] , X) =\tilde{g}([Z^{c},Y^{c}],X^{c})=0$, which shows that $F$ is of Douglas type.
\end{proof}
\begin{lem}\label{M1}
Assume that $F$ is an arbitrary left invariant $ (\alpha , \beta) $-metric defined by
\ref{me}. Then $F^{v}$ is of Berwald type if and only if $ad^{\ast}_{X} = ad_{X} $ and for any $Y\in\frak{g}$, $\nabla _{X}Y = \frac{1}{2} [X,Y]$.
\end{lem}
\begin{proof}
$ F^{v} $ is of Berwald type if and only if $ \tilde{\nabla} _{Y^{c}}X^{v} = \tilde{\nabla} _{Y^{v}}X^{v} =0$. Now formula \ref{Levi-Civita connection} completes the proof.
\end{proof}
\begin{remark}
If we consider $F$ is of Berwald type then, the previous lemma together with the formula \ref{Levi-Civita connection} show that, $F^{v}$ is of Berwald type if and only if $ X \in z(\mathfrak{g}) $, where $ z(\mathfrak{g})$ denotes the center of $ \mathfrak{g} $.\\
\end{remark}
\begin{lem}\label{22}
Let $F$ be a left  invariant Randers metric on a Lie group $G$. Then $F$ is a Douglas metric if and only if $ F^{v} $ is a Douglas metric.
\end{lem}
\begin{proof}
If $F$ is a Douglas metric of Randers type, then by theorem 3.2 of \cite{An-Deng}, for all $Y , Z \in \mathfrak{g}$,
$g ([ Z ,Y] , X) =0$. So $ F^{v} $ is a Douglas metric because
\begin{equation}
   \tilde{g}([Z^{v},Y^{v}],X^{v})=0,~~\tilde{g}([Z^{c},Y^{c}],X^{v})=0,~~\tilde{g}([Z^{v},Y^{c}],X^{v})= g ([ Z ,Y] , X).
\end{equation}
Conversely, let $ F^{v} $ be a Douglas metric. Then the above equations show that $F$ is a Douglas metric because by theorem 3.2 of  \cite{An-Deng} we have
\begin{equation}
   \tilde{g}([Z^{v},Y^{c}],X^{v})=0,~~\forall Y , Z \in \mathfrak{g}.
\end{equation}
\end{proof}

In the following theorems we give the flag curvature formulas of  $ F^{c} $ and  $ F^{v} $
explicitly, where $F$ is of  Douglas type.
\begin{theorem}\label{M2}
Suppose that $G$ is a Lie group equipped with a left invariant Riemannian metric $g$ and
\begin{equation*}
    F = \sqrt{g(y,y)} \phi(\frac{g(X(x),y)}{\sqrt{g(y,y)}}),
\end{equation*}
is a left invariant $(\alpha,\beta)$-metric of Berwald type on $G$. Then for the flag curvature of the left invariant $(\alpha,\beta)$-metric metric $F^{c}$ on $TG$ we have: \\
\begin{enumerate}
  \item $\tilde{P} =\textsf{span}\lbrace Y^{c} ,V^{c} \rbrace$,
        \begin{equation*}
            K^{F^{c}}( \tilde{P},Y^{c}) =\frac{1}{\phi^{2}(g(X,Y)).[(1+ g^{2}(X,V)D)]} K(V,Y),
        \end{equation*}
  \item $\tilde{P} =\textsf{span} \lbrace Y^{c} ,V^{v} \rbrace$,
        \begin{eqnarray*}
          K^{F^{c}}( \tilde{P},Y^{c}) = \frac{1}{\phi^{2}(g(X,Y))} \lbrace &&\hspace*{-0.6cm} K(V,Y) +\frac{1}{2} g([V, \nabla_{Y}V ],Y) - \frac{1}{2} g(\nabla_{V}ad^{\ast}_{V}Y, Y) \\
          &&+ \frac{1}{4} g([V, ad^{\ast}_{V}Y],Y) - \frac{1}{2} g([[Y,V], V], Y) \rbrace,
        \end{eqnarray*}
  \item $\tilde{P} =\textsf{span} \lbrace Y^{v} ,V^{c} \rbrace$,
        \begin{eqnarray*}
          K^{F^{c}}( \tilde{P},Y^{v}) = \frac{1}{\phi^{2}(0).[(1+ g^{2}(X,V)D)]}\lbrace&&\hspace*{-0.6cm}
        K(V,Y) +\frac{1}{2} g([Y, \nabla_{V}Y ],U) - \frac{1}{2} g(\nabla_{Y}ad^{\ast}_{Y}V, V) \\
          &&+ \frac{1}{4} g([Y, ad^{\ast}_{Y}V],V) - \frac{1}{2} g([[V,Y], Y], V) \rbrace,
        \end{eqnarray*}
  \item $ \tilde{P} =\textsf{span} \lbrace Y^{v} ,V^{v} \rbrace$,
        \begin{equation*}
            K^{F^{c}}( \tilde{P},Y^{v}) =\frac{1}{\phi^{2}(0)} \lbrace K(V,Y) + g( \nabla_{[V,Y]}Y ,V) + \frac{1}{4} \parallel [V,Y] \parallel^{2} \rbrace,
        \end{equation*}
\end{enumerate}
where $K^{F^{c}}$ and $K$ denote the flag curvature of $F^{c}$ and the sectional curvature of $g$, respectively, and $\lbrace Y,V \rbrace$ is an orthonormal basis for the two dimensional subspace $P$ of $\mathfrak{g}$, with respect to $g$, and $D=\frac{\phi''}{\phi - s\phi'}$.
\end{theorem}
\begin{proof}
Lemma \ref{fd} shows that  $F^{c}$ is of Berwald type. Therefore, the Chern connection of $F^{c}$ and the Levi-Civita connection of $\tilde{g}$ coincide. Now by using theorem 2.4 of \cite{Asgari-Salimi 1} and the flag curvature formula given in proposition 3.2 of \cite{Deng-Hosseini-Liu-Salimi} the proof is completed.
\end{proof}
\begin{theorem}\label{M3}
Assume that $G$ is a Lie group equipped with a left invariant Riemannian metric $g$. Suppose that
\begin{equation}
    F = \sqrt{g(y,y)} + g(X(x),y)
\end{equation}
is the induced left invariant Randers Metric of Douglas type on $G$ which is defined by $g$ and a left invariant vector field $X$. Then for the flag curvature of the left invariant Randers metric $F^{c} $ on $TG$ we have: \\
\begin{enumerate}
  \item $ \tilde{P} =\textsf{span} \lbrace Y^{c} ,V^{c} \rbrace $,
        \begin{eqnarray*}
            K^{F^{c}}(\tilde{P},Y^{c}) = \frac{1}{(1+g(X,Y))^{2}}K(V,Y) +&& \hspace*{-0.6cm}\frac{1}{4(1+g(X,Y))^{2}} \lbrace 3 g([X,Y],Y)\\
            && -4(1+g(X,Y))g(U(Y,\Sigma_{i=1}^{m} \eta_{i}X_{i}),X)\rbrace,
        \end{eqnarray*}

  \item $ \tilde{P} =\textsf{span}\lbrace Y^{c} ,U^{v} \rbrace$,
        \begin{eqnarray*}
          K^{F^{c}}( \tilde{P},Y^{c}) = \frac{1}{(1+g(X,Y))^{2}} \lbrace &&\hspace*{-0.6cm} K(V,Y) +\frac{1}{2} g([V, \nabla_{Y}V ],Y) - \frac{1}{2} g(\nabla_{V}ad^{\ast}_{V}Y, Y)\\
          && +\frac{1}{4} g([V, ad^{\ast}_{V}Y],Y) - \frac{1}{2} g([[Y,V], V], Y) \rbrace \\
          && +\frac{1}{4(1+g(X,Y))^{2}} \lbrace 3 g^{2}([X,Y],Y)\\
          && -4(1+g(X,Y))g(U(Y,\Sigma_{i=1}^{m} \eta_{i}X_{i}),X)\rbrace,
        \end{eqnarray*}
  \item $ \tilde{P} =\textsf{span} \lbrace Y^{v} ,V^{c} \rbrace$,
        \begin{eqnarray*}
          K^{F^{c}}( \tilde{P},Y^{v}) = K(V,Y) + && \hspace*{-0.6cm} \frac{1}{2} g([Y, \nabla_{V}Y ],V) - \frac{1}{2} g(\nabla_{Y}ad^{\ast}_{Y}V, V) \\
          && + \frac{1}{4} g([Y, ad^{\ast}_{Y}V],V) - \frac{1}{2} g([[V,Y], Y], V) \\
          && + \frac{1}{4}\lbrace 3g^{2}([Y,X],Y)+4g(U(Y,\Sigma_{i=1}^{m} \mu_{j}X_{j}),X)\rbrace,
        \end{eqnarray*}
  \item $ \tilde{P} =\textsf{span} \lbrace Y^{v} ,V^{v} \rbrace$,
        \begin{eqnarray*}
          K^{F^{c}}(\tilde{P},Y^{v})= K(V,Y) + && \hspace*{-0.6cm} g( \nabla_{[V,Y]}Y ,V) + \frac{1}{4} \parallel [V,Y] \parallel^{2}\\
          && + \frac{1}{4}\lbrace 3g^{2}([Y,X],Y)+4g(U(Y,\Sigma_{i=1}^{m} \mu_{j}X_{j}),X)\rbrace,
        \end{eqnarray*}
\end{enumerate}
where $\lbrace X_{i}| i=1,\cdots, m \rbrace$ is a basis for the Lie algebra $\mathfrak{g}$ of $G$ and $ U:\mathfrak{g} \times \mathfrak{g} \rightarrow \mathfrak{g} $ is a symmetric function defined by the following equation,
\begin{equation}\label{111}
 2g(U(v_{1},v_{2}),v_{3})=g([v_{3},v_{1}],v_{2})+g([v_{3},v_{2}],v_{1}).
\end{equation}
\end{theorem}
\begin{proof}
Lemma \ref{fd} shows that $F^{c}$ is of Douglas type. It is sufficient to use theorem 2.4 of \cite{Asgari-Salimi 1} and the following formula of the flag curvature which is given in theorem 2.1 of \cite{Deng-Hu},
$$ K^{F^{c}}(P,Y^{c}) = \frac{\tilde{g}(Y^{c},Y^{c})}{F^{c}(Y^{c})^{2}}\tilde{K}(\tilde{P})+\frac{1}{4 F^{c}(Y^{c})^{4}} \lbrace 3\tilde{g}(\tilde{U}(Y^{c},Y^{c}),X^{c})-4 F \tilde{g}(\tilde{U}(Y^{c},\tilde{U}(Y^{c},Y^{c}),X^{c}),$$
where $\tilde{U}:\tilde{\mathfrak{g}} \times \tilde{\mathfrak{g}}  \rightarrow \tilde{\mathfrak{g}}=Lie(TG) $ satisfy in the formula \ref{111}. So
\begin{eqnarray}
\begin{array}{lr}
\tilde{g}(\tilde{U}(Y^{c},Y^{c}),X^{c})=g([X,Y],Y),~~\tilde{g}(\tilde{U}(Y^{c},\tilde{U}(Y^{c},Y^{c})),X^{c}) = g(U(Y,\Sigma_{i=1}^{m}\eta_{i}X_{i}),X)), &\nonumber\\
\tilde{g}(\tilde{U}(Y^{v},Y^{v}),X^{c})=g([Y,X],Y),~~\tilde{g}(\tilde{U}(Y^{v},\tilde{U}(Y^{v},Y^{v})),X^{c}) = -g(U(Y,\Sigma_{j=1}^{m}\mu_{j}X_{j}), X),  &
\end{array}
\end{eqnarray}
where $\tilde{U}(Y^{c},Y^{c}) = \Sigma_{i=1}^{m}\eta_{i}X_{i}^{c} +\Sigma_{i=1}^{m}\delta_{i}X_{i}^{v}$ and $\tilde{U}(Y^{v},Y^{v}) = \Sigma_{j=1}^{m}\lambda_{j}X_{j}^{c} +\Sigma_{j=1}^{m}\mu_{j}X_{j}^{v}$.
\end{proof}
In the following theorems we compute the flag curvature of $F^{v}$.
\begin{theorem}
Suppose that $G$ is a Lie group equipped with a left invariant Riemannian metric $g$. If
\begin{equation}
    F = \sqrt{g(y,y)} \phi(\frac{g(X(x),y)}{\sqrt{g(y,y)}}),
\end{equation}
is a left invariant $(\alpha,\beta)$-metric on $G$ such that $F^{v}$ is of Berwald type. Then for the flag curvature of the left invariant $(\alpha,\beta)$-metric $F^{v}$ on $TG$ we have:
\begin{enumerate}
    \item $\tilde{P} =\textsf{span} \lbrace Y^{c} ,V^{c} \rbrace $,
        \begin{equation*}
            K^{F^{v}}( \tilde{P},Y^{c}) =\frac{1}{\phi^{2}(0)} K(V,Y),
        \end{equation*}
    \item $\tilde{P} =\textsf{span} \lbrace Y^{c} ,V^{v} \rbrace$,
        \begin{eqnarray*}
          K^{F^{v}}( \tilde{P},Y^{c}) = \frac{1}{\phi^{2}(0).[1 + g^{2}(X,V)D]} \lbrace &&\hspace*{-0.6cm} K(V,Y) +\frac{1}{2} g([V, \nabla_{Y}V],Y) - \frac{1}{2} g(\nabla_{V}ad^{\ast}_{V}Y, Y) \\
          && +\frac{1}{4} g([V, ad^{\ast}_{V}Y],Y) - \frac{1}{2} g([[Y,V], V], Y) \rbrace,
        \end{eqnarray*}
    \item $\tilde{P} =\textsf{span} \lbrace Y^{v} ,V^{c} \rbrace$,
        \begin{eqnarray*}
          K^{F^{v}}( \tilde{P},Y^{v}) = \frac{1}{\phi^{2}(g(X,Y))}\lbrace &&\hspace*{-0.6cm} K(V,Y) +\frac{1}{2} g([Y, \nabla_{V}Y ],U) - \frac{1}{2} g(\nabla_{Y}ad^{\ast}_{Y}V, V)\\
          && +\frac{1}{4} g([Y, ad^{\ast}_{Y}V],V) - \frac{1}{2} g([[V,Y], Y], V) \rbrace,
        \end{eqnarray*}
    \item $\tilde{P} =\textsf{span} \lbrace Y^{v} ,V^{v} \rbrace$,
        \begin{eqnarray*}
          K^{F^{v}}( \tilde{P},Y^{v}) =\frac{1}{\phi^{2}(g(X,Y)).[1 + g^{2}(X,V)D]} \lbrace K(V,Y) + g( \nabla_{[V,Y]}Y ,V) + \frac{1}{4} \parallel [V,Y] \parallel^{2} \rbrace,
        \end{eqnarray*}
\end{enumerate}
where $K^{F^{v}}$ and $K$ denote the flag curvature of $F^{v}$ and the sectional curvature of $g$ respectively, and $\lbrace Y,V\rbrace$ is an orthonormal basis for $P$ with respect to $g$.
\end{theorem}
\begin{proof}
It is sufficient to use theorem 2.4  of \cite{Asgari-Salimi 1} and the curvature formula of proposition 3.2 of \cite{Deng-Hosseini-Liu-Salimi}.
\end{proof}
\begin{theorem}
Let $G$ be a Lie group equipped with a left invariant Riemannian metric $g$ and
\begin{equation}
    F = \sqrt{g(y,y)} + g(X(x),y),
\end{equation}
is a left invariant Randers metric of Douglas type on $G$ defied by $g$ and a left invariant vector field $X$. Then for the flag curvature of the left invariant $(\alpha,\beta)$-metric $F^{v}$ on $TG$ we have:
\begin{enumerate}
  \item $\tilde{P} =\textsf{span} \lbrace Y^{c} ,V^{c} \rbrace$,
    \begin{equation*}
        K^{F^{v}}( \tilde{P},Y^{c}) = K(V,Y) - \frac{1}{2}g([X,Y], \Sigma_{j=1}^{m} \delta_{j}X_{j}),
    \end{equation*}
  \item $\tilde{P} =\textsf{span} \lbrace Y^{c} ,V^{v} \rbrace$,
     \begin{eqnarray*}
       K^{F^{v}}( \tilde{P},Y^{c}) =  K(V,Y) +&&\hspace*{-0.6cm}\frac{1}{2} g([V, \nabla_{Y}V],Y) - \frac{1}{2} g(\nabla_{V}ad^{\ast}_{V}Y, Y) \\
       && +\frac{1}{4} g([V, ad^{\ast}_{V}Y],Y) - \frac{1}{2} g([[Y,V], V], Y) - \frac{1}{2}g([X,Y], \Sigma_{j=1}^{m} \delta_{j}X_{j}),
     \end{eqnarray*}
  \item $\tilde{P} =\textsf{span} \lbrace Y^{v} ,U^{c} \rbrace$,
    \begin{eqnarray*}
      K^{F^{v}}( \tilde{P},Y^{v}) =&&\hspace*{-0.6cm}\frac{1}{(1+g(X,y))^{2}} \lbrace K(U,Y) +\frac{1}{2} g([Y, \nabla_{U}Y ],U) - \frac{1}{2} g(\nabla_{Y}ad^{\ast}_{Y}U, U) \\
      &&\hspace*{2cm} +\frac{1}{4} g([Y, ad^{\ast}_{Y}U],U)-\frac{1}{2} g([[U,Y], Y], U) \rbrace\\
      &&\hspace*{-0.6cm}-\frac{1}{2(1+g(X,y))^{4}}g([X,\Sigma_{j=1}^{m} \eta_{j}X_{j}], Y),
    \end{eqnarray*}
  \item $\tilde{P} =\textsf{span} \lbrace Y^{v} ,U^{v} \rbrace$,
    \begin{eqnarray*}
      K^{F^{v}}( \tilde{P},Y^{v})=&&\hspace*{-0.6cm}\frac{1}{(1+g(X,y))^{2}} \lbrace K(U,Y) + g( \nabla_{[U,Y]}Y ,U)+ \frac{1}{4} \parallel [U,Y] \parallel^{2} \rbrace \\
      &&\hspace*{-0.6cm}- \frac{1}{2(1+g(X,y))^{4}}g([X,\Sigma_{j=1}^{m} \eta_{j}X_{j}], Y).
    \end{eqnarray*}
\end{enumerate}
\end{theorem}
\begin{proof}
Lemma \ref{22} shows that $F^{v}$ is of Douglas type. We know that
\begin{eqnarray}
\begin{array}{lr}
\tilde{g}(\tilde{U}(Y^{c},Y^{c}),X^{v})=\tilde{g}(\tilde{U}(Y^{v},Y^{v}),X^{v})=0, &\nonumber\\
\tilde{g}(\tilde{U}(Y^{c},\tilde{U}(Y^{c},Y^{c})),X^{v})=\frac{1}{2}g([X,Y], \Sigma_{j=1}^{m} \delta_{j}X_{j}), &\nonumber\\
\tilde{g}(\tilde{U}(Y^{v},\tilde{U}(Y^{v},Y^{v})),X^{v})=\frac{1}{2}g([X,\Sigma_{j=1}^{m} \eta_{j}X_{j}], Y).&
\end{array}
\end{eqnarray}
Now a similar method to the proof of theorem \ref{M3} completes the proof.
\end{proof}


\section{\textbf{Examples}}
In this section we study the flag curvature of two important families of $(\alpha,\beta)$-metrics which are called Matsumoto and Kropina metrics. Similar to the Randers metric, these metrics have physical application (see \cite{Antonelli-Ingarden-Matsumoto} and \cite{Matsumoto2}).
\begin{example}
Let $G$ be a Lie group equipped with a left invariant Riemannian metric $g$ and
\begin{equation}\label{Matsumoto metric}
    F = \frac{g(y,y)}{\sqrt{g(y,y)} - g(X(x),y)}
\end{equation}
be the Berwaldian left invariant Matsumoto metric on $G$, defined by $g$ and a left invariant vector field $X$ which is parallel with respect to the Levi-civita connection of $g$. Then for the flag curvature of the left invariant Matsumoto Metric $F^{c} $ on $TG$ we have:
\begin{enumerate}
  \item $ \tilde{P} =\textsf{span}\lbrace Y^{c} ,U^{c} \rbrace $,
    \begin{eqnarray*}
      K^{F^{c}}( \tilde{P},Y^{c}) = \frac{(1-g(X,Y))^{3}(1-2g(X,Y))}{1+2 g^{2}(X, U) +2g^{2}(X, Y)-3g(X,Y)} K(U,Y),
    \end{eqnarray*}
  \item $ \tilde{P} =\textsf{span}\lbrace Y^{c} ,U^{v} \rbrace$,
    \begin{eqnarray*}
      K^{F^{c}}( \tilde{P},Y^{c}) =&&\hspace*{-0.6cm} (1-g(X,Y))^{2} \lbrace K(U,Y) +\frac{1}{2} g([U, \nabla_{Y}U ],Y) \\
      &&\hspace*{-0.6cm}-\frac{1}{2} g(\nabla_{U}ad^{\ast}_{U}Y, Y) + \frac{1}{4} g([U, ad^{\ast}_{U}Y],Y) - \frac{1}{2} g([[Y,U], U], Y) \rbrace,
    \end{eqnarray*}
  \item $ \tilde{P} =\textsf{span}\lbrace Y^{v} ,U^{c} \rbrace$,
    \begin{eqnarray*}
      K^{F^{c}}( \tilde{P},Y^{v}) =&&\hspace*{-0.6cm} \frac{1}{2 g^{2}(X,U)+1} \lbrace
        K(U,Y) +\frac{1}{2} g([Y, \nabla_{U}Y ],U)\\
        &&\hspace*{-0.6cm}- \frac{1}{2} g(\nabla_{Y}ad^{\ast}_{Y}U, U) + \frac{1}{4} g([Y, ad^{\ast}_{Y}U],U) - \frac{1}{2} g([[U,Y], Y], U) \rbrace,
    \end{eqnarray*}
  \item $ \tilde{P} =\textsf{span}\lbrace Y^{v} ,U^{v} \rbrace$,
    \begin{eqnarray*}
      K^{F^{c}}( \tilde{P},Y^{v}) = K(U,Y) + g( \nabla_{[U,Y]}Y ,U) + \frac{1}{4} \parallel [U,Y] \parallel^{2},
    \end{eqnarray*}
\end{enumerate}
where the assumptions are similar to the previous section and $\nabla$ denotes the Levi-Civita connection of $g$.
In this case, the formulas for the flag curvature of $F^v$ are as follows:
\begin{enumerate}
  \item $ \tilde{P} =\textsf{span}\lbrace Y^{c} ,U^{c} \rbrace $,
    \begin{eqnarray*}
          K^{F^{v}}( \tilde{P},Y^{c}) = K(U,Y),
    \end{eqnarray*}
  \item $ \tilde{P} =\textsf{span}\lbrace Y^{c} ,U^{v} \rbrace$,
    \begin{eqnarray*}
          K^{F^{v}}( \tilde{P},Y^{c}) =&&\hspace*{-0.6cm}\frac{1}{1+2 g^{2}(X,U)} \lbrace K(U,Y) +\frac{1}{2} g([U, \nabla_{Y}U ],Y) \\
          &&\hspace*{-0.6cm}- \frac{1}{2} g(\nabla_{U}ad^{\ast}_{U}Y, Y) + \frac{1}{4} g([U, ad^{\ast}_{U}Y],Y) - \frac{1}{2} g([[Y,U], U], Y) \rbrace,
    \end{eqnarray*}
  \item $ \tilde{P} =\textsf{span}\lbrace Y^{v} ,U^{c} \rbrace$,
    \begin{eqnarray*}
          K^{F^{v}}( \tilde{P},Y^{v}) = &&\hspace*{-0.6cm}(1- g(X,Y))^{2} \lbrace
          K(U,Y) +\frac{1}{2} g([Y, \nabla_{U}Y ],U) \\
            &&\hspace*{-0.6cm}- \frac{1}{2} g(\nabla_{Y}ad^{\ast}_{Y}U, U) + \frac{1}{4} g([Y, ad^{\ast}_{Y}U],U) - \frac{1}{2} g([[U,Y], Y], U) \rbrace,
    \end{eqnarray*}
  \item $ \tilde{P} =\textsf{span}\lbrace Y^{v} ,U^{v} \rbrace$,
    \begin{eqnarray*}
          K^{F^{v}}( \tilde{P},Y^{v})=&&\hspace*{-0.6cm}\frac{(1-g(X,Y))^{3}(1-2g(X,Y))}{1+2 g^{2}(X, U) +2g^{2}(X, Y)-3g(X,Y)} \\
          &&\hspace*{-0.6cm}\lbrace K(U,Y) + g( \nabla_{[U,Y]}Y ,U) + \frac{1}{4} \parallel [U,Y] \parallel^{2} \rbrace.
    \end{eqnarray*}
\end{enumerate}
\end{example}

\begin{example}
Suppose that $G$ is a Lie group equipped with a left invariant Riemannian metric $g$ and
\begin{equation}\label{Kropina metric}
    F = \frac{g(y,y)}{ g(X(x),y)},
\end{equation}
is a left invariant Kropina metric of Berwald type on $G$ defined by $g$ and a left invariant vector field $X$. Then for the flag curvature of the left invariant Kropina Metric $ F^{c} $ on $TG$ we have:
\begin{enumerate}
  \item $ \tilde{P} =\textsf{span}\lbrace Y^{c} ,U^{c} \rbrace $,
    \begin{eqnarray*}
        K^{F^{c}}( \tilde{P},Y^{c}) =\frac{g^{4}(X,Y)}{g^{2}(X,U) + g^{2}(X,Y)} K(U,Y),
    \end{eqnarray*}
  \item $ \tilde{P} =\textsf{span}\lbrace Y^{c} ,U^{v} \rbrace$,
    \begin{eqnarray*}
          K^{F^{c}}( \tilde{P},Y^{c}) =&&\hspace*{-0.6cm} (g(X,Y))^{2} \lbrace K(U,Y) +\frac{1}{2} g([U, \nabla_{Y}U ],Y) \\
          &&\hspace*{-0.6cm} -\frac{1}{2} g(\nabla_{U}ad^{\ast}_{U}Y, Y) + \frac{1}{4} g([U, ad^{\ast}_{U}Y],Y) - \frac{1}{2} g([[Y,U], U], Y) \rbrace,
    \end{eqnarray*}
  \item $ \tilde{P} =\textsf{span}\lbrace Y^{v} ,U^{c} \rbrace$
    \begin{eqnarray*}
          K^{F^{c}}(\tilde{P},Y^{v})~ is~ not ~defined,
    \end{eqnarray*}
  \item $ \tilde{P} =\textsf{span}\lbrace Y^{v} ,U^{v} \rbrace$,
    \begin{eqnarray*}
           K^{F^{c}}( \tilde{P},Y^{v})~ is~ not ~defined.
    \end{eqnarray*}
\end{enumerate}
Also the flag curvature formulas of the Finsler metric $F^v$ are as follows:
\begin{enumerate}
  \item $ \tilde{P} =\textsf{span}\lbrace Y^{c} ,U^{c} \rbrace $,
    \begin{eqnarray*}
        K^{F^{v}}( \tilde{P},Y^{c})~ is~ not~defined,
    \end{eqnarray*}
  \item $ \tilde{P} =\textsf{span}\lbrace Y^{c} ,U^{v} \rbrace$,
    \begin{eqnarray*}
          K^{F^{v}}( \tilde{P},Y^{c})~ is~~not~~defined,
    \end{eqnarray*}
  \item $ \tilde{P} =\textsf{span}\lbrace Y^{v} ,U^{c} \rbrace$
    \begin{eqnarray*}
          K^{F^{v}}( \tilde{P},Y^{v}) =&&\hspace*{-0.6cm} g^{2}(X,Y) \lbrace K(U,Y) +\frac{1}{2} g([Y, \nabla_{U}Y ],U) \\
            &&\hspace*{-0.6cm}-\frac{1}{2} g(\nabla_{Y}ad^{\ast}_{Y}U, U) + \frac{1}{4} g([Y, ad^{\ast}_{Y}U],U) - \frac{1}{2} g([[U,Y], Y], U) \rbrace,
    \end{eqnarray*}
  \item $ \tilde{P} =\textsf{span}\lbrace Y^{v} ,U^{v} \rbrace$,
    \begin{eqnarray*}
           K^{F^{v}}( \tilde{P},Y^{v})=\frac{g^{4}(X,Y)}{g^{2}(X,U) + g^{2}(X,Y)} \lbrace K(U,Y) + g( \nabla_{[U,Y]}Y ,U) + \frac{1}{4} \parallel [U,Y] \parallel^{2} \rbrace.
    \end{eqnarray*}
\end{enumerate}
\end{example}


\end{document}